\documentclass[12pt]{amsart}

\usepackage[margin=2.5cm]{geometry}
\usepackage[T1]{fontenc}
\usepackage[utf8]{inputenc}
\usepackage{times}
\usepackage{microtype}
\usepackage{amsmath,amssymb,amsthm,mathtools}
\usepackage[shortlabels]{enumitem}
\usepackage{tikz-cd}
\usepackage[
  backend=biber,
  style=alphabetic,
  maxnames=50,
  minalphanames=4,
  maxalphanames=4
]{biblatex}
\usepackage[
  colorlinks=false,
  pdfborder={0 0 0.8},
  linkbordercolor={0.9 0.05 0.05},
  citebordercolor={0.05 0.65 0.05},
  urlbordercolor={0 0.68 0.82},
  filebordercolor={0 0.68 0.82}
]{hyperref}

\newcommand\FF{\mathbb{F}}

\newcommand{\spec}{\operatorname{Spec}}

\newcommand{\cofib}{\operatorname{cofib}}

\renewcommand{\to}[1][]{\xrightarrow{\ #1\ }}

\title{Derived Kunz and Radu--Andr\'e Theorems}
\author{Daniel Fink}

\theoremstyle{plain}
\newtheorem{theorem}{Theorem}[section]

\newtheorem{lemma}[theorem]{Lemma}
\newtheorem{corollary}[theorem]{Corollary}

\theoremstyle{definition}

\begin{document}

\begin{abstract}
We extend Kunz's theorem \cite{Kun69} and its relative form, the Radu--Andr\'e theorem \cite{Rad92,And93}, to animated algebras. The main inputs are the classical theorems, together with Bhatt and Scholze's vanishing of the Frobenius on negative cohomology.
\end{abstract}

\maketitle
\thispagestyle{empty}

\section{Introduction}

Kunz's theorem characterizes regularity for noetherian rings in prime characteristic in terms of the Frobenius endomorphism. It says that a noetherian ring $R$ of prime characteristic $p$ is regular if and only if its absolute Frobenius
\[
  F_R:R\to R,\qquad r\mapsto r^p,
\]
is flat \cite{Kun69}. 

There is a relative form of this criterion, due to Radu and Andr\'e. If
$f:R\to S$ is a homomorphism of noetherian rings of characteristic $p$, then
$f$ is regular, that is, flat with geometrically regular fibers, if and only if
the ordinary relative Frobenius
\[
  F_{S/R,\heartsuit} : S\otimes_{R,F_R} R \to S,
  \qquad s\otimes r\mapsto s^p f(r),
\]
is flat \cite{Rad92,And93}. 

The purpose of this note is to show that these two results extend to animated rings. Besides the classical theorems, the main input is the result of Bhatt and Scholze that the Frobenius of an animated $\FF_p$-algebra induces the zero map on negative cohomology. The relative case is then proved using a fiberwise flatness criterion and induction on the cohomological truncation degree.

\subsection*{Acknowledgments}
I would like to thank Manuel Blickle for many helpful discussions and for his comments on an earlier version of this manuscript.

The author acknowledges support from the Deutsche Forschungsgemeinschaft (DFG, German Research Foundation) through the Collaborative Research Centre TRR~326 \textit{Geometry and Arithmetic of Uniformized Structures}, project number 444845124.

\subsection*{Declaration on computer assistance}
The proof of Lemma~\ref{lemma:key-lemma} was developed with computer assistance, specifically through discussions with ChatGPT and Claude.

\section{Derived Kunz and Radu--Andr\'e Theorems}

We recall the following notions. Let $R$ be an animated ring and let $M\in D(R)^{\le0}$ be a connective $R$-module. Then $M$ is called \emph{flat} if $H^0(M)$ is flat as an ordinary $H^0(R)$-module, and the natural map $H^k(R)\otimes_{H^0(R)}H^0(M)\to H^k(M)$ is an isomorphism for every $k\le -1$.

A map $f:R\to S$ of animated rings is called \emph{flat} if its target $S$ is flat as a module over its source $R$. We say that $f$ is \emph{regular} if it is flat and the induced map $H^0(R)\to H^0(S)$ of ordinary rings is regular \cite[\href{https://stacks.math.columbia.edu/tag/07BZ}{Tag 07BZ}]{stacks-project}.

We will use the following result that the Frobenius acts trivially on negative cohomology. The argument given here differs from the original one \cite[Proposition~11.6]{BS17}.

\begin{lemma}
\label{lemma:Frobenius-negative-cohomology-paper}
For any animated $\FF_p$-algebra $R$, the absolute Frobenius map $F_R$ induces the zero map on negative cohomology groups.
\end{lemma}
\begin{proof}
Let $k\le -1$. We show that $H^{k}(F_R)$ is the zero map. Let $S = \tau^{\ge k+1} R$ denote the cohomological $k+1$-truncation of $R$, and let $\varphi: R\to S$ be the natural truncation map. Let $\varepsilon_{\varphi}:\cofib(\varphi)\otimes^L_R S \to L_{S/R}$ denote the Hurewicz map \cite[\S 25.3.6]{Lur18}. The naturality square of the Frobenius endomorphisms associated with $\varphi$ induces a commutative diagram:
\[
\begin{tikzcd}
\cofib(\varphi) \arrow[r] \arrow[d] &  \cofib(\varphi)\otimes^L_R S \arrow[d] \arrow[r, "\varepsilon_{\varphi}"] & L_{S/R} \arrow[d] \\
\cofib(\varphi) \arrow[r]           & \cofib(\varphi)\otimes^L_R S \arrow[r, "\varepsilon_{\varphi}"]           & L_{S/R}\rlap{.}
\end{tikzcd}
\]
The vertical arrows are induced by the absolute Frobenius maps. The left-hand square is a diagram in $D(R)$ and the right-hand square in $D(S)$. The lower row carries the module structure induced from the upper row by restriction of scalars along the Frobenius. Thus, the right-hand vertical morphism factors as a composition
\[
L_{S/R} \longrightarrow L_{S/R}\otimes^L_{S,F} S \xrightarrow{0} L_{S/R},
\]
in which the second arrow is zero by the proof of \cite[Lemma~2.9]{Fin25}. This shows that the right-hand vertical arrow is zero. Now, applying $H^{k-1}$ to the outer square yields the right-hand square in the following commutative diagram of $H^0(R)$-modules:
\[
\begin{tikzcd}
H^k(R) \arrow[d, "H^k(F)"'] &
H^{k-1}(\cofib \varphi) \arrow[l, "\cong"'] \arrow[d] \arrow[r] &
H^{k-1}(L_{S/R}) \arrow[d, "0"] \\
H^k(R) &
H^{k-1}(\cofib \varphi) \arrow[l, "\cong"] \arrow[r] &
H^{k-1}(L_{S/R}) \rlap{.}
\end{tikzcd}
\]
By \cite[Remark~25.3.6.5]{Lur18}, the right-hand horizontal maps in this diagram are isomorphisms, and hence $H^k(F_R)=0$.
\end{proof}

The derived Kunz theorem follows immediately from this lemma and the classical Kunz theorem.

\begin{theorem}[Derived Kunz Theorem]
  \label{theorem:derived kunz}
  Let $R$ be an animated $\FF_p$-algebra, and suppose that $H^0(R)$ is noetherian. Then $R$ is discrete and regular if and only if its Frobenius endomorphism is flat.
\end{theorem}
\begin{proof}
  The ``only if'' direction follows from Kunz's theorem. Conversely, if $F_R$ is flat, then $R$ is discrete by Lemma~\ref{lemma:Frobenius-negative-cohomology-paper} and the definition of flatness; see also \cite[Lemma~2.6]{Fin25}. If the discrete ring $R$ is noetherian, then Kunz's theorem implies that it is regular.
\end{proof}

Recall that $M$ is flat if and only if $M \otimes^L_R H^0(R)$ is a flat $H^0(R)$-module, in which case the latter canonically identifies with $H^0(M)$. Indeed, the only-if assertion follows from the fact that flatness is preserved under base change \cite[Proposition~7.2.2.16(1)]{Lur17}. Conversely, by \cite[Theorem~7.2.2.15]{Lur17}, the module $M$ is flat if and only if, for every discrete module $N\in D(R)^{\heartsuit}$, the derived tensor product $N\otimes^L_{R}M$ is discrete. By \cite[Proposition~7.1.1.13]{Lur17}, any discrete $R$-module is equivalent to a discrete $H^0(R)$-module, so we have an equivalence
\[
N\otimes^L_{R}M\simeq N\otimes^L_{H^0(R)}(M\otimes^L_{R}H^0(R)),
\]
which shows the discreteness of the left-hand side if $M \otimes^L_R H^0(R)$ is flat over $H^0(R)$.

This extends the flatness criterion of Avramov--Foxby
\cite[5.3.F. Proposition]{AF91}, in the form stated in
\cite[Lemma~9.5.2]{BLM21}, to the animated setting. This extension was
already established in \cite[Lemma~3.7.4]{Lur04}, where it was used to prove
a derived Popescu theorem. It will also be used in the proof of the derived
Radu--Andr\'e theorem.

\begin{lemma}
  \label{lemma:avramov-foxby}
  Let $R$ be an animated ring, and suppose that $H^0(R)$ is noetherian. Let $M\in D(R)^{\le 0}$ be a connective $R$-module. Then $M$ is flat if and only if $M\otimes^L_R \kappa(x)$ is discrete for every $x\in \spec(H^0(R))$.
\end{lemma}

We will also need the following lemma. We say that an animated ring $R$ is cohomologically \emph{$n$-truncated}, for some $n\le 0$, if $H^k(R)=0$ for all $k\le n-1$. 

\begin{lemma}
  \label{lemma:key-lemma}
  Let $n\le -1$, and let $f\colon R\to S$ be a map of cohomologically $n$-truncated animated $\FF_p$-algebras. If $\tau^{\ge n+1}f$ and $\tau^{\ge n}F_{S/R}$ are flat, then $f$ is flat.
\end{lemma}
\begin{proof}
First, for a cohomologically $n$-truncated animated $\FF_p$-algebra $A$, consider the fiber sequence
$
H^n(A)[n]\to A \to \tau^{\ge n+1} A
$
in $D(A)$. Base change along the first factor map $A\to A\otimes^L_{A,F}A$ yields a fiber sequence
\[
(H^n(A)\otimes^L_{A,F}A)[n]\to
A\otimes^L_{A,F}A \to
(\tau^{\ge n+1}A)\otimes^L_{A,F}A
\]
in $D(A\otimes^L_{A,F}A)$. Since $-\otimes^L_{A,F}A$ is right $t$-exact for the standard $t$-structures, this fiber sequence induces an exact sequence of $H^0(A\otimes^L_{A,F}A)$-modules:
\[
H^n(A)\otimes_{H^0(A),F}H^0(A)
\to H^n(A\otimes^L_{A,F}A)
\to H^n((\tau^{\ge n+1}A)\otimes^L_{A,F}A)\to 0.
\]
The first map in this exact sequence is zero. To see this, apply $H^n$ to the following commutative diagram and use Lemma~\ref{lemma:Frobenius-negative-cohomology-paper} together with the $H^0(A)$-linearity of the second factor map:
\[
\begin{tikzcd}
H^n(A) \arrow[d] \arrow[r]              & A \arrow[d] \arrow[rd, "F"]   &    \\
{(H^n(A)\otimes^L_{A,F}A)[n]} \arrow[r] & {A\otimes^L_{A,F}A} \arrow[r, "\simeq"] & {A\rlap{.}}
\end{tikzcd}
\]
Thus, the second map in the exact sequence is an isomorphism. As the second factor map $\iota_2:A \to A\otimes_{A,F}^L A$ is an equivalence, we obtain isomorphisms
\[
  F_*H^n(A)
  \to[\cong]
  H^n(A\otimes^L_{A,F}A)
  \to[\cong]
  H^n(\tau^{\ge n+1}A\otimes^L_{A,F}A).
  \]
Under the above equivalence $A \simeq A\otimes^L_{A,F} A$, the $H^0(A)$-module structure on $H^n(A)$ induced by the right tensor factor is the usual one, while the one induced by the left tensor factor is obtained by restriction along the Frobenius endomorphism.

Now, consider the map $f:R \to S$. As $\tau^{\ge n+1}f$ is flat by assumption, it suffices to show that the natural map 
\[
  c_S:H^n(R)\otimes_{H^0(R)}H^0(S) \to H^n(S)
  \]
  is an isomorphism. The above argument for $A=S$ yields an isomorphism
\[
  H^n(\iota_2):H^n(S)\to[\cong] H^n(\tau^{\ge n+1}S\otimes^L_{S,F}S)
\]
of $H^0(S)$-modules, with respect to the second factor map. Moreover, we have the following chain of isomorphisms:
  \[
  \begin{aligned}
    H^n(\tau^{\ge n+1}S\otimes^L_{S,F}S)
    &\cong H^n(\tau^{\ge n+1}S\otimes^L_{S}S\otimes^L_{R,F}R \otimes^L_{S\otimes^L_{R,F}R}S)\\
    &\xleftarrow{\cong} H^n(\tau^{\ge n+1}S\otimes^L_{S}S\otimes^L_{R,F}R) \otimes_{H^0(S\otimes^L_{R,F}R)}H^0(S)\\
    &\cong H^n(\tau^{\ge n+1}S\otimes^L_{\tau^{\ge n+1}R}\tau^{\ge n+1}R\otimes^L_{R,F}R) \otimes_{H^0(S\otimes^L_{R,F}R)}H^0(S)\\
    &\xleftarrow{\cong} H^0(S)\otimes_{H^0(R)}H^n(\tau^{\ge n+1}R\otimes^L_{R,F}R) \otimes_{H^0(S\otimes^L_{R,F}R)}H^0(S)\\
    &\xleftarrow{\cong}  H^0(S)\otimes_{H^0(R),F}H^0(R)\otimes_{H^0(R)} H^n(R) \otimes_{H^0(S\otimes^L_{R,F}R)}H^0(S)\\
    &\cong H^n(R)\otimes_{H^0(R)}H^0(S).
  \end{aligned}
  \]
  The first isomorphism holds because the absolute Frobenius $F_S$ factors as a composition
  $S\to S\otimes^L_{R,F}R \to[F_{S/R}] S$; the second because $\tau^{\ge n}F_{S/R}$ is flat by assumption; the third because the composite $R \to S \to \tau^{\ge n+1} S$ is canonically equivalent to $R \to \tau^{\ge n+1}R \to \tau^{\ge n+1} S$; the fourth because $\tau^{\ge n+1} R \to \tau^{\ge n+1} S$ is flat by assumption; the fifth by the above with $A=R$, using the module structure induced by the first factor map; and the last because the action of $H^0(S\otimes^L_{R,F}R)$ on
  $H^0(S)\otimes_{H^0(R),F}H^0(R)\otimes_{H^0(R)} H^n(R)$
  is through the first two tensor factors via the canonical isomorphism
  $H^0(S\otimes^L_{R,F}R)\cong H^0(S)\otimes_{H^0(R),F}H^0(R)$, and because the map $R\to S$ factors through $F_{S/R}$. In total, we obtain a natural isomorphism of $H^0(S)$-modules
  \[
  \theta_S: H^n(S)\to[\cong] H^n(R)\otimes_{H^0(R)}H^0(S).
  \]
  It remains to show the identity $\theta_{S}^{-1}=c_S$ where $c_S$ is the natural map from above. As $\theta_{S}^{-1}$ is $H^0(S)$-linear, it suffices to show that $\theta_{S}^{-1}(x\otimes 1)=H^n(f)(x)$ for every $x\in H^n(R)$. Tracing back the above chain of equivalences, the element $x\otimes 1$ is sent to $H^n(\iota_2)(H^n(f)(x))$ which is precisely the image of $H^n(f)(x)$ under the isomorphism above with $A=S$. This finishes the proof.
\end{proof}

For $n=0$, the equivalence of (a) and (c) in the following theorem is precisely the classical Radu--Andr\'e theorem.

\begin{theorem}[Derived Radu--Andr\'e Theorem]
  \label{theorem:derived Radu-Andre}
  Let $f:R\to S$ be a map of animated $\FF_p$-algebras, and suppose that $H^0(R)$ and $H^0(S)$ are noetherian rings. Let $n\le 0$. Consider the following statements.
  \begin{enumerate}[(a)]
    \item The map $f$ is regular.
    \item The relative Frobenius map $F_{S/R}$ is flat.
    \item The cohomologically $n$-truncated relative Frobenius map $\tau^{\ge n} F_{S/R}$ is flat.
  \end{enumerate}
  Then (a) and (b) are equivalent and imply (c). Moreover, if $R$ and $S$ are cohomologically $n$-truncated, then~(c) implies~(a).
\end{theorem}
\begin{proof}
  We first show the equivalence of (a) and (b). Assume that $f$ is regular. Then $f$ is flat, and hence so is its base change $R\to S\otimes^L_{R,F}R$. Thus, $F_{S/R}$ is a map of flat $R$-algebras, and, for $k\le -1$, the canonical maps are isomorphisms:
  \[
  H^k(S\otimes^L_{R,F}R)\otimes_{H^0(S\otimes^L_{R,F}R)}H^0(S)\cong H^k(R)\otimes_{H^0(R)}H^0(S\otimes^L_{R,F}R)\otimes_{H^0(S\otimes^L_{R,F}R)}H^0(S) \cong H^k(S).
  \]
  Moreover, the map $H^0(R)\to H^0(S)$ of ordinary noetherian rings is regular and thus $H^0(F_{S/R})=F_{H^0(S)/H^0(R),\heartsuit}$ is flat by the Radu--Andr\'e theorem. This proves (b). Conversely, suppose that $F_{S/R}$ is flat. Let $x\in \spec(H^0(R))$, and let $\kappa/\kappa(x)$ be a finite field extension of the residue field at $x$. Then $H^0(S\otimes^L_R \kappa)$ is noetherian. Since
  \[
    F_{S\otimes^L_R \kappa/\kappa}\simeq F_{S/R}\otimes^L_R \kappa
  \]
  is flat, the absolute Frobenius $F_{S\otimes^L_R \kappa}$ is flat as the composite of flat maps
  \[
    S\otimes^L_R \kappa
    \to[\phantom{F_{S\otimes^L_R \kappa/\kappa}}]
    (S\otimes^L_R \kappa)\otimes^L_{\kappa,F}\kappa
    \to[F_{S\otimes^L_R \kappa/\kappa}]
    S\otimes^L_R \kappa.
  \]
  Thus $S\otimes^L_R \kappa$ is discrete and regular by Theorem~\ref{theorem:derived kunz}. Taking $\kappa=\kappa(x)$ shows that $f$ is flat by Lemma~\ref{lemma:avramov-foxby}, and hence regular.

  It is clear that (b) implies (c). We next assume that $R$ and $S$ are cohomologically $n$-truncated and that $\tau^{\ge n} F_{S/R}$ is flat. We prove (a) by induction on $n$, with the base case $n=0$ being the Radu--Andr\'e theorem. Suppose that $n\le -1$. Since $\tau^{\ge n} F_{S/R}$ is flat, its cohomological $(n+1)$-truncation
  \[
  \tau^{\ge n+1} F_{S/R} \simeq \tau^{\ge n+1} F_{\tau^{\ge n+1}S/\tau^{\ge n+1}R}
  \]
  is also flat, and thus $\tau^{\ge n+1}f$ is regular by the inductive hypothesis. It remains to show that $f$ is flat. This follows from Lemma~\ref{lemma:key-lemma}, since $\tau^{\ge n} F_{S/R}$ and $\tau^{\ge n+1}f$ are flat.
\end{proof}

  The same argument as in the proof of the implication \textup{(b)} $\Rightarrow$ \textup{(a)} in Theorem~\ref{theorem:derived Radu-Andre} shows the following statement.

\begin{corollary}
  Let $f:R\to S$ be a map of animated $\FF_p$-algebras, and suppose that $H^0(R)$ is noetherian. If $F_{S/R}$ is flat, then $f$ is flat and all its geometric fibers have a flat absolute Frobenius.
\end{corollary}

If $H^0(R)$ is not noetherian, then the previous statement need not be true. 
If $R=\FF_p[x^{1/p^{\infty}}]$ denotes the perfect polynomial ring in one variable, then the relative Frobenius of the projection $R\to \FF_p$ with kernel $(x^{1/p^{\infty}})$ is an isomorphism but the projection not flat.

\emergencystretch=1em
\printbibliography

\end{document}